\documentclass[11pt]{article}
\usepackage{latexsym}
\usepackage{amssymb}
\usepackage{amsmath}
\usepackage{amsthm}
\newtheorem{thm}{Theorem}
\newtheorem{lem}[thm]{Lemma}

\newtheorem{rem}[thm]{Remark}

\newtheorem{pro}[thm]{Proposition}
\newtheorem{assumption}[thm]{Assumption}

\newcommand{\RR}{{\mathbb R}}

\newcommand{\ep}{\varepsilon}
\newcommand{\la}{\lambda}

\newcommand{\Ricbar}{\overline{\text{\rm Ric}(\gamma)}_t}

\newcommand{\DVbar}{\overline{\nabla V_y^{\la}(t,\gamma)}_t}
\newcommand{\grad}{\text{\rm grad}}
\newcommand{\cxy}{c_{x,y}}
\newcommand{\cyx}{c_{y,x}}
\newcommand{\Rreverse}{R^{\leftarrow}}
\renewcommand{\Im}{{\rm Image}}
\newcommand{\esssup}{{\rm esssup}}
\newcommand{\dirichlet}{e^{\lambda}_{Dir,2,P_{x,y}({\mathcal D})}}
\numberwithin{equation}{section}
\begin{document}
\newcounter{aaa}
\newcounter{bbb}
\newcounter{ccc}
\newcounter{ddd}
\newcounter{eee}
\newcounter{xxx}
\newcounter{xvi}
\newcounter{x}
\setcounter{aaa}{1}
\setcounter{bbb}{2}
\setcounter{ccc}{3}
\setcounter{ddd}{4}
\setcounter{eee}{32}
\setcounter{xxx}{10}
\setcounter{xvi}{16}
\setcounter{x}{38}
\title
{Semi-classical limit of the generalized second lowest eigenvalue of 
Dirichlet Laplacians on small domains in path spaces}
\author{Shigeki Aida\\
Mathematical Institute\\
Tohoku University,
Sendai, 980-8578, JAPAN\\
e-mail: aida@math.tohoku.ac.jp}
\date{}
\maketitle
\begin{abstract}
Let $M$ be a complete Riemannian manifold.
Let $P_{x,y}(M)$ be the space of continuous paths on
$M$ with fixed starting point $x$ and ending point $y$.
Assume that $x$ and $y$ is close enough such that
the minimal geodesic $\cxy$ between $x$ and $y$
is unique.
Let $-L_{\la}$ be the Ornstein-Uhlenbeck operator with
the Dirichlet boundary condition on a 
small neighborhood of the geodesic 
$\cxy$ in $P_{x,y}(M)$.
The underlying measure $\bar{\nu}^{\la}_{x,y}$
of the $L^2$-space is 
the normalized probability measure of the restriction
of the pinned Brownian motion measure on the neighborhood 
of $\cxy$ and 
$\la^{-1}$ is the variance parameter of the 
Brownian motion.
We show that the generalized second lowest
eigenvalue of $-L_{\la}$ divided by $\la$
converges to the lowest eigenvalue of the
Hessian of the energy function 
of the $H^1$-paths at $\cxy$
under the small variance limit (semi-classical limit)
$\la\to\infty$.
\end{abstract}

\section{Introduction}
Let $(M,g)$ be an $n$-dimensional complete Riemannian manifold.
Let $x,y\in M$.
Let $\nu_{x,y}^{\la}$ be the pinned Brownian motion measure on the
pinned path space $P_{x,y}(M)=C([0,1]\to M~|~\gamma(0)=x,\gamma(1)=y)$,
where $\la$ is a positive parameter such that
the transition probability of the Brownian motion is given by
$p(t/\la,x,y)$, where $p(t,x,y)$ denotes the heat kernel
of the $L^2$-semigroup $e^{t\Delta/2}$ on $L^2(M,dx)$.
Here $dx$ denotes the Riemannian volume and
$\Delta$ is the Laplace-Beltrami operator.
The parameter $\la$ coincides with the inverse number of the variance 
parameter of the Brownian motion on a Euclidean space
which is obtained by the stochastic-development of the Brownian motion on
$M$ to the tangent space $T_xM$ at the starting point $x$.
Let $D_0$ be the $H$-derivative on $P_{x,y}(M)$ and consider the
Dirichlet form in $L^2(P_{x,y}(M),d\nu^{\la}_{x,y})$:
\begin{equation}
{\mathcal E}(F,F)=\int_{P_{x,y}(M)}|D_0F(\gamma)|^2d\nu^{\la}_{x,y}.
\end{equation}
Note that the $H$-derivative is defined by using the Levi-Civita
connection.
Let $-L_{\la}$ be the non-negative generator of the above Dirichlet form.
This is a generalization of the usual Ornstein-Uhlenbeck operator
in the case where $M$ is a Euclidean space.
We are interested in the spectral properties of the operator
$-L_{\la}$. 
For example, are there eigenvalues?, are there spectral gaps,
where the essential spectrum is ?, and so on.
In particular, the limit $\la\to\infty$ is a kind of semi-classical
limit since formally the Brownian motion measure is written as
$$
d\nu^{\la}_{x,y}(\gamma)=\frac{1}{Z_{\la}}
\exp\left(-\frac{\la}{2}\int_0^1|\gamma'(t)|^2dt\right)
d\gamma.
$$
Therefore one may expect that the asymptotic behavior of the low-lying spectrum
of $-L_{\la}$ is related with the critical points of 
the energy function $E(\gamma)=\frac{1}{2}\int_0^1|\gamma'(t)|^2dt$, that is,
the set of geodesics.
However, of course, this naive observation is not true in general since
the space of paths $P_{x,y}(M)$ is non-compact and unbounded set and
we need to put some assumptions at the infinity.
Hence we study the Ornstein-Uhlenbeck operators on a certain domain
with Dirichlet boundary condition in this paper
which still captures essential points of the problem.
We explain what kind of subsets of $P_{x,y}(M)$ we are interested in.
Let ${\mathcal D}=B_r(y)=\{z\in M~|~d(z,y)<r\}$,
where $d$ stands for the Riemannian distance function on $(M,g)$.
We assume that $x\in {\mathcal D}$ and the closure of ${\mathcal D}$ does not 
contain the cut-locus of $y$.
Actually we put stronger assumptions on ${\mathcal D}$
which we will explain in the next section.
Let $P_{x,y}({\mathcal D})$ be the all paths $\gamma\in P_{x,y}(M)$
satisfying $\gamma(t)\in {\mathcal D}$~$(0\le t\le 1)$.
Let $\cxy=\cxy(t)$~$(0\le t\le 1)$ be 
the uniquely defined minimal geodesic between
$x$ and $y$.
Then $\cxy\in P_{x,y}({\mathcal D})$ and
the set $P_{x,y}({\mathcal D})$ is an open neighborhood of 
$\cxy$.
We study the Ornstein-Uhlenbeck operator $-L_{\la}$ with
Dirichlet boundary condition
(we call just it a Dirichlet Laplacian) 
on $L^2(P_{x,y}({\mathcal D}),d\bar{\nu}^{\la}_{x,y})$,
where $\bar{\nu}^{\la}_{x,y}$ is the normalized probability measure 
of the restriction of $\nu^{\la}_{x,y}$ to $P_{x,y}({\mathcal D})$.
It is well-known that the pinned measure converges weakly to the atomic measure
$\delta_{\cxy}$ at $\cxy$.
This implies that the bottom of spectrum of the Dirichlet Laplacian
converges to $0$.
The aim of this paper is to show that the second generalized 
lowest eigenvalue divided by $\la$ converges to
the bottom of the spectrum of the Hessian of the energy function
$E(\gamma)$ at $\cxy$ as $\la\to \infty$
which is naturally conjectured by an
analogy of semi-classical analysis in finite dimensional cases.
We note that a rough lower and upper bound was 
already given by Eberle~\cite{eberle3}.
See also \cite{eberle1,eberle2}.
Also I gave a rough lower bound by using a Clark-Ocone-Haussman
(=COH) formula
in \cite{aida-coh2}.
Our result is a refinement of such results and the proof is based on
the COH formula.

The paper is organized as follows.
In Section 2, 
we explain the assumptions on the Riemannian manifold and the domain
${\mathcal D}$.
Next we give the definitions of Dirichlet Laplacian 
, the second generalized lowest eigenvalue and
we state our main theorem.
Also, we recall our COH formula.
In Section 3, we identify the semi-classical limit of the
coefficient operator $A(\gamma)_{\la}$ in the COH formula.
To this end, the Jacobi fields along the geodesic
$\cxy$ play important roles.
In particular we show that the limit of
$A(\gamma)_{\la}$ is related with the square root of the Hessian of
the energy function of $E$ at $\cxy$.
This observation is important in our proof.
After these preparations, we prove our main theorem in
Section 4.

\section{Statement of results}
We assume that the Ricci curvature of $(M,g)$
is bounded which implies the non-explosion of the
Brownian motion.
Actually, our problem is local one and this assumption is not 
important for our problem.
We mention this local property later.

The Sobolev space 
$H^{1,2}(P_{x,y}(M),\nu^{\la}_{x,y})$ is the completion of
the vector spaces of smooth cylindrical functions
by the $H^1$-Sobolev norm:
$$
\|F\|_{H^1}=\left(\|F\|_{L^2(\nu^{\la}_{x,y})}^2+
{\mathcal E}(F,F)\right)^{1/2},
$$
where ${\mathcal E}$ is the Dirichlet form which is defined in the
introduction.
As already defined, let ${\cal D}$ be a metric open ball 
centered at $y$ with radius $r$.
Throughout this paper, we always assume the
following.
\begin{assumption}
$(1)$~Let us denote the set of cut-locus of $y$
by ${\mathrm Cut}(y)$.
Then there are no intersection of the closure of
${\cal D}$ and ${\mathrm Cut}(y)$.
Also $x\in {\mathcal D}$.

\noindent
$(2)$
The Hessian of $k(z)=\frac{1}{2}d(z,y)^2$
satisfies that
$
\inf_{z\in {\cal D}}\nabla^2k(z)>1/2.
$
\end{assumption}
Since $\nabla_z^2k(z)|_{z=y}=I_{T_yM}$,
the above assumption (1), (2) hold in a small domain containing
$y$.
Under the above assumptions, clearly 
the minimal geodesic $\cxy=\cxy(t)$~
$(0\le t\le 1)$~$(\cxy(0)=x, \cxy(1)=y)$
belong to ${\mathcal D}$.
Let
\begin{equation}
P_{x,y}\left({\mathcal D}\right)=\left\{\gamma\in P_{x,y}(M)~|~
\mbox{$\gamma(t)\in {\mathcal D}$ for all $0\le t\le 1$}\right\}
\end{equation}
which is an open neighborhood of $\cxy$ in $P_{x,y}(M)$.
By the assumption, there are no geodesics other than $\cxy$
in $P_{x,y}({\mathcal D})$.
Let
\begin{equation}
H^{1,2}_0(P_{x,y}({\mathcal D}),\nu^{\la}_{x,y})=
\left\{F\in H^{1,2}(P_{x,y}(M),\nu^{\la}_{x,y})~|~
\mbox{$F=0$ outside $P_{x,y}({\mathcal D})$}
\right\}
\end{equation}
which is a closed linear subspace of $H^{1,2}(P_{x,y}(M),\nu^{\la}_{x,y})$.
We denote the normalized probability 
$d\nu^{\la}_{x,y}/\nu^{\la}_{x,y}\left(P_{x,y}({\mathcal D})\right)$
on $P_{x,y}({\mathcal D})$ by
$d\bar{\nu}^{\la}_{x,y}$.
The non-positive generator $L_{\la}$ corresponding to the 
densely defined closed form 
$$
{\cal E}(F,F),~~F\in H^{1,2}_0(P_{x,y}({\mathcal D}),\bar{\nu}^{\la}_{x,y})
$$
in the Hilbert space
$L^2(P_{x,y}({\mathcal D}),\bar{\nu}^{\la}_{x,y})$
is the Dirichlet Laplacian on 
$P_{x,y}({\mathcal D})$.
Let 
\begin{eqnarray}
e^{\la}_{Dir,1,P_{x,y}({\mathcal D})}&=&
\inf_{F(\ne 0)\in H^{1,2}_0\left(P_{x,y}({\mathcal D})\right)}
\frac{\int_{P_{x,y}({\mathcal D})}|D_0F|^2d\bar{\nu}^{\la}_{x,y}}
{\|F\|_{L^2(\bar{\nu}^{\la}_{x,y})}^2}.
\end{eqnarray}
This is equal to $\inf\sigma(-L_{\la})$, where
$\sigma(-L_{\la})$ denotes the spectral set of $-L_{\la}$.
Next we introduce 
\begin{eqnarray}
\lefteqn{\dirichlet}\nonumber\\
& &=
\sup_{G(\ne 0)\in L^2(\bar{\nu}^{\la}_{x,y})}
\inf\Biggl\{
\frac{\int_{P_{x,y}({\mathcal D})}|D_0F|^2d\bar{\nu}^{\la}_{x,y}}
{\|F\|_{L^2(\bar{\nu}^{\la}_{x,y})}^2}~\Bigg |~
F\in H^{1,2}_0\left(P_{x,y}({\mathcal D})\right), \nonumber\\
& &\qquad\qquad (F,G)_{L^2(\bar{\nu}^{\la}_{x,y})}=0\Biggr\}
\end{eqnarray}
which is the generalized second lowest eigenvalue of
$-L_{\la}$.
Let $P_{x,y}(M)^{H^1}$ be the subset of $P_{x,y}(M)$
consisting of $H^1$-paths.
Let us consider the energy function of $H^1$-path:
\begin{equation}
E(\gamma)=\frac{1}{2}\int_0^1|\gamma'(t)|^2dt
\qquad \gamma\in P_{x,y}(M)^{H^1}.\label{energy function}
\end{equation}
We use the same notation $D_0$ for the $H$-derivative
of the smooth function on $P_{x,y}(M)^{H^1}$.
The following is our main theorem.

\begin{thm}\label{main theorem}
We have
\begin{equation}
\lim_{\la\to\infty}
\frac{\dirichlet}{\la}=e_0,
\end{equation}
where
$e_0=\inf\sigma((D_0^2 E)(\cxy))$.
\end{thm}

\begin{rem}
If the sectional curvature on each points
of the geodesic $\cxy$ is positive, then
$\inf\sigma(D_0^2 E(\cxy))<1$ 
and the bottom of the spectrum is an eigenvalue
and is not an essential spectrum.
While the curvature is strictly negative, 
$\inf\sigma(D_0^2 E(\cxy))=1$
and $1$ is not an eigenvalue and belongs to essential spectrum.
This suggests that the second lowest eigenvalue, or more generally,
some low-lying spectrum of the Dirichlet Laplacian 
on $P_{x,y}({\mathcal D})$
over a positively curved manifold belongs to the discrete spectrum.
Also if some isometry group acts $M$ with the fixed points
$x$ and $y$, we may expect the discrete spectrum have some
multiplicities.
We show these kind of results in the case where $M$ is a compact
Lie group in a forthcoming paper.
As for general Riemannian manifold cases, we need more works.
\end{rem}

In the proof of Theorem~\ref{main theorem}, we use 
a short time behavior of logarithmic derivative of heat kernels
which is due to Malliavin and Stroock.

\begin{thm}[Malliavin-Stroock~\cite{ms}]\label{malliavin-stroock}
Assume that $M$ is compact and let $z\in \mathrm{Cut}(y)^c$.
Then
\begin{equation}
\lim_{t\to 0}t\nabla^2_z\log p(t,y,z)=-\nabla^2_zk(z)
\label{short time asymptotics}
\end{equation}
uniformly on any compact subset of $\mathrm{Cut}(y)^c$.
\end{thm}

Clearly, the same result holds on ${\mathbb R}^n$ with a
Riemannian metric which coincides with the Euclidean metric
outside a certain compact subset.
Of course, 
similar result might hold in more general non-compact cases but
the result for the perturbation of the Euclidean metric
is enough for our purpose because
we are concerned with Dirichlet Laplacian on a small domain.
We explain this meaning more precisely.
We already consider a metric ball $B_r(y)$ which includes
$x$ in $M$.
In addition to $(M,g)$,
let $(M',g')$ be another Riemannian manifold and 
$x',y'$ be points on $M'$.
Let $d'$ be the distance function on $(M',g')$.
Let $r_{\ast}$ be a positive number which is greater than $r$.
We denote by $B_{r_{\ast}}(y')$ the metric ball centered at $y'$
with radius $r_{\ast}$ in $M'$.
Assume that $B_{r_{\ast}}(y')$ is diffeomorphic to an open 
Euclidean ball.
Also assume that $x'\in B_{r_{\ast}}(y')$ and
there exists an bijective Riemannian isometry $\Phi$
from $B_{r_{\ast}}(y)$ to $B_{r_{\ast}}(y')$ such that
$\Phi(x)=x'$ and $\Phi(y)=y'$.
Let us define $P_{x',y'}(B_{r}(y'))$ and
the normalized probability measure 
$\bar{\nu}^{\la}_{x',y'}$ on $P_{x',y'}(B_{r}(y'))$
in a similar manner to
$P_{x,y}(B_r(y))$ and so on.
Define a mapping $\Psi$ from $P_{x,y}(B_{r_{\ast}}(y))$ to
$P_{x',y'}(B_{r_{\ast}}(y')))$ by $(\Psi\gamma)(t)=\Phi(\gamma(t))$~
$\gamma\in P_{x,y}(B_{r_{\ast}}(y))$.
Note that the image of $P_{x,y}(B_r(y))$ by $\Psi$ is
exactly $P_{x',y'}(B_r(y'))$.
Then using the uniqueness of the solution of the stochastic differential equation,
we see that
\begin{itemize}
\item[(1)]~$\Psi : (P_{x,y}(B_r(y)),\bar{\nu}^{\la}_{x,y})
\to (P_{x',y'}(B_r(y')), \bar{\nu}^{\la}_{x',y'})$
and its inverse map are measure preserving map.
\item[(2)]~Let $F\in H^{1,2}(P_{x',y'}(M'))$.
If $F\in H^{1,2}_0(P_{x',y'}(B_r(y')))$, then
$$
\tilde{F}(\gamma):=
(F\circ \Psi)(\gamma) \chi(\sup_{0\le t\le 1}d'(\Psi(\gamma)(t),y))
\in H^{1,2}_0(P_{x,y}(B_r(y))),
$$
where $\chi=\chi(t)$ is a non-negative smooth function such that
$\chi(t)=1$ for $t\le \frac{r+r_{\ast}}{2}$ and
$\chi(t)=0$ for $t\ge \frac{r+2r_{\ast}}{3}$.
Moreover $\|F\|_{L^2(\bar{\nu}^{\la}_{x',y'})}
=\|\tilde{F}\|_{L^2(\bar{\nu}^{\la}_{x,y})}$
and $\|D_0F\|_{L^2(\bar{\nu}^{\la}_{x',y'})}
=\|D_0\tilde{F}\|_{L^2(\bar{\nu}^{\la}_{x,y})}$.
\end{itemize}
The above observation implies that
$$
e^{\la}_{Dir,2,P_{x,y}(B_r(y))}=e^{\la}_{Dir,2,P_{x',y'}(B_r(y'))}.
$$
Hence, there are some freedom of varying the Riemannian metric
outside a certain compact subset to study our problem.
Hence, we may assume that 
$M$ is diffeomorphic to ${\mathbb R}^n$
and the Riemannian metric is flat outside a certain bounded
subset and the short time asymptotics in (\ref{short time asymptotics})
holds.
The key ingredient of the proof of Theorem~\ref{main theorem}
is a version of the Clark-Ocone-Haussman formula in
\cite{aida-coh2} which can be extended to the above non-compact
${\mathbb R}^n$ case with a nice Riemannian metric.
Since the formula is strongly related with the heat kernel $p(t,x,y)$ 
on $M$ itself, the above observation is important.

Let us recall the COH formula in
\cite{aida-coh2}.
See also \cite{chl}, \cite{gong-ma}, \cite{aida-coh}.
Let ${\mathfrak F}_t=\sigma(\{\gamma(s)~|~0\le s\le t\})$.
Let
$\tau(\gamma)_t : T_xM\to T_{\gamma(t)}M$ 
be the stochastic parallel translation along the semi-martingale
$\gamma(t)$ under $\nu^{\la}_x$ which is defined by the Levi-Civita
connection.
Then $b(t)=\int_0^t\tau(\gamma)_s^{-1}\circ d\gamma(s)$
is an ${\mathfrak F}_t$-Brownian motion with the covariance
$E^{\nu^{\la}_x}[(b(t),u)(b(s),v)]=(u,v)\frac{t\wedge s}{\la}$
~$(u,v\in T_xM)$ on $T_xM$ under $\nu^{\la}_x$.
We recall the notion of the trivialization.
Let $T\in \Gamma((\otimes^pTM)\otimes(\otimes^qT^{\ast}M))$ be a
$(p,q)$-tensor on $M$.
Using the (stochastic)-parallel translation
$\tau(\gamma)_t : T_xM\to T_{\gamma(t)}M$,
$T(\gamma(t))\in (\otimes^pT_{\gamma(t)}M)\otimes
(\otimes^qT_{\gamma(t)}^{\ast}M)$
can be naturally identified with
a tensor in $\otimes^pT_{x}M\otimes^qT_x^{\ast}M$
which we denote by $\overline{T(\gamma)}_t$.
For example, for the time dependent vector field
$V_y^{\la}(t,z)=\grad_z\log p\left(\frac{1-t}{\la},y,z\right)$
~$(0\le t<1)$
we denote
$$
\overline{V_y^{\la}(t,\gamma)}_t=
\tau(\gamma)_t^{-1}V_y^{\la}(t,\gamma(t))
\in T_xM.
$$
Also we denote
$\DVbar=\tau(\gamma)_t^{-1}\nabla_{z}V_y^{\la}(t,z)|_{z=\gamma(t)}$.
More explicitly,
\begin{equation}
\DVbar=
\tau(\gamma)_t^{-1}\nabla_z\grad_z \log p\left(\frac{1-t}{\la},y,z\right)
\Bigg |_{z=\gamma(t)}\tau(\gamma)_t.
\end{equation}

Let 
$w(t)=b(t)-\frac{1}{\la}\int_0^t\overline{V_y^{\la}(s,\gamma)}_sds$.
This process is defined for $t<1$ and it is not difficult to check that
this can be extended continuously up to $t=1$.
Let $\mathcal{N}^{x,y,t}$ be the set of all null sets of
$\nu_{x,y}|_{{\mathfrak F}_t}$ and set
${\mathfrak G}_t={\mathfrak F}_t\vee \mathcal{N}^{x,y,1}$.
Then $w$ is an ${\mathfrak G}_t$-adapted Brownian
motion for $0\le t\le 1$ such that
$E^{\nu^{\la}_{x,y}}
[\left(w(t),u\right)_{T_xM}\left(w(s),v\right)_{T_xM}]
=\frac{t\wedge s}{\la}(u,v)$ for any
$u,v\in T_xM$.
Let
\begin{equation}
K(\gamma)_{\la,t}=
-\frac{1}{2\la}\Ricbar
+\frac{1}{\la}\DVbar.
\end{equation}
Let $M(\gamma)_{\la,t}$ be the linear mapping on $T_xM$ 
satisfying the differential equation:
\begin{eqnarray}
M(\gamma)_{\la,t}'&=&K(\gamma)_{\la,t}M(\gamma)_{\la,t}
\quad 0\le t\le 1,\\
M(\gamma)_{\la,0}&=&I.
\end{eqnarray}
Using $M$ and $K$, we define
\begin{eqnarray*}
J(\gamma)_{\la}\varphi(t)&=&
(M(\gamma)_{\la,t}^{\ast})^{-1}\int_t^1
M(\gamma)_{\la,s}^{\ast}K(\gamma)_{\la,s}\varphi(s)ds.
\end{eqnarray*}
Then $J(\gamma)_{\la}$ is a bounded linear operator 
on $L^2$ for every $\gamma$.
Also let 
\begin{equation}
A(\gamma)_{\la}=I_{L^2}+J(\gamma)_{\la}.
\end{equation}

Now, we state our Clark-Ocone-Haussman formula for
$F\in H^{1,2}_0(P_{x,y}({\mathcal D}))$ 
and its immediate consequences.
We refer the readers to \cite{clw1} \cite{clw2}
for further related
developments in analysis in path spaces.

\begin{thm}\label{coh formula}~

\noindent
$(1)$~Let
\begin{equation}
\xi_{\la}=\esssup\left\{\|A(\gamma)_{\la}\|_{op}~|~\gamma\in
P_{x,y}({\cal D})\right\}.
\end{equation}
Then there exists $\la_0>0$ such that
$\sup_{\la\ge \la_0}\xi_{\la}<\infty$.

\noindent
$(2)$~For any $F\in H^{1,2}_0(P_{x,y}({\mathcal D}))$,
$D_0F(\gamma)=0$ for $\nu^{\la}_{x,y}$-almost all 
$\gamma\in P_{x,y}({\mathcal D})^c$.

\noindent
$(3)$~
Let $\la\ge \la_0$.
For any $F\in H^{1,2}_0(P_{x,y}({\mathcal D}))$,
the following COH formula holds:
$$
E^{\nu^{\la}_{x,y}}[F|{\mathfrak G}_t]=
E^{\nu^{\la}_{x,y}}[F]+
\int_0^t\left(H(s,\gamma),dw(s)\right),\quad 0\le t\le 1,
$$
where
\begin{equation}
H(s,\gamma)=
E^{\nu^{\la}_{x,y}}\left[
A(\gamma)_{\la}(D_0F(\gamma)')(s)| {\mathfrak G}_s\right].
\quad\label{coh1}
\end{equation}
$(\ref{coh1})$ denotes the predictable projection
and $D_0F(\gamma)'_t=\frac{d}{dt}(D_0F)(\gamma)_t$.

\noindent
$(4)$~
The following inequalities hold.

\begin{eqnarray}
E^{\bar{\nu}^{\la}_{x,y}}\left[
F^2\log \left(F^2/\|F\|^2_{L^2(\bar{\nu}_{x,y}^{\la})}\right)
\right]
&\le&
\frac{2\xi_{\la}}{\la}E^{\bar{\nu}^{\la}_{x,y}}\left[
|D_0F|^2\right],\label{lsi1}\\
\frac{\la}{\xi_{\la}}
E^{\bar{\nu}^{\la}_{x,y}}\left[\left(F-E^{{\nu}^{\la}_{x,y}}[F]\right)^2
\right]
&\le& E^{\bar{\nu}^{\la}_{x,y}}\left[
|D_0F|^2\right].\label{poincare1}
\end{eqnarray}
\end{thm}

(\ref{poincare1}) implies that $e^{\la}_{Dir,2,P_{x,y}({\mathcal D})}$
diverges to $+\infty$ at least of the order of $O(\la)$.
Therefore, $e^{\la}_{Dir,1,P_{x,y}({\mathcal D})}$
is a simple eigenvalue.
We denote the normalized non-negative
eigenfunction (ground state function)
by $\Omega$.
It is clear that $\Omega\in H^{1,2}_0(P_{x,y}({\mathcal D}),
\bar{\nu}^{\la}_{x,y})$.
Also we have
\begin{eqnarray}
e^{\la}_{Dir,2,P_{x,y}({\mathcal D})}&=&
\inf\Biggl\{
\frac{\int_{P_{x,y}({\mathcal D})}
|D_0(F-(\Omega,F)\Omega)|^2d\bar{\nu}^{\la}_{x,y}}
{\|F-(\Omega,F)\Omega\|_{L^2(\bar{\nu}^{\la}_{x,y})}^2}
~\Bigg |
~F\in H^{1,2}_0(P_{x,y}(\mathcal D))~\nonumber\\
& &\qquad\mbox{and}~
\|F-(\Omega,F)\Omega\|_{L^2(\bar{\nu}^{\la}_{x,y})}\ne 0
\Biggr\}.
\end{eqnarray}
It is plausible that $\Omega$ is a strictly positive for
$\nu^{\la}_{x,y}$ almost all $\gamma$
which 
follows from the 
positivity improving property of the corresponding
$L^2$-semigroup.
However, we do not need such a property in this paper
and we do not consider such a problem.

\section{Square root of Hessian of the energy function and
Jacobi fields}

Let $\xi$ be the tangent vector at $x$ such that
$\exp_x(t\xi)=\cxy(t)$~$(0\le t\le 1)$, where
$\exp_x$ stands for the exponential mapping at $x$.
Clearly it holds that $d(x,y)=\|\xi\|_{T_xM}$.
We denote $\cyx(t)=\cxy(1-t)$ which is a reverse geodesic path from
$y$ to $x$.
By the assumption that
$x$ is not in the cut-locus of $y$,
when $\la\to\infty$, the pinned Brownian motion measure converges weakly
to the atomic measure $\delta_{\cxy}$ at $\cxy$.
Also by Theorem~\ref{malliavin-stroock},
when $\gamma$ and $\cxy$ are close and $\la$ is large enough,
$K(\gamma)_{\la,t}$ can be approximated with
$K(t)$ 
where
\begin{equation}
K(t)=-\frac{1}{1-t}\overline{\nabla^2k(\cxy)}_t.
\end{equation}
Hence when $\la\to\infty$,
the coefficient operator $I+J(\gamma)_{\la}$ in the COH formula
converges to the corresponding
non-random $I+J_0$ (or $(S^{-1})^{\ast}$) 
which is defined using the Hessian of 
the square of the distance function $k$ along the geodesic $\cxy$.
In fact, this observation leads our main results.
We refer the readers for the notations $J_0$ and 
$(S^{-1})^{\ast}$ to Lemma~\ref{S explicit form} and
Lemma~\ref{perturbation of M}.
In order to see the explicit expression of
the Hessian of
$k(z)$~$(z\in \cxy)$, we recall the notion of
Jacobi fields.

Let $R$ be the curvature tensor and define $R(t)=
\overline{R(\cxy)}_t(\cdot,\xi)(\xi)$
which is a linear mapping on $T_xM$.
Also we define $\Rreverse(t)=R(1-t)$.
Let $v\in T_xM$ and
$W(t,v)$ be the solution to the following ODE:
\begin{equation}
W''(t,v)+\Rreverse(t)W(t,v)=0~~0\le t\le 1,
~~
W(0,v)=0,~ W'(0,v)=v.
\end{equation}
Since $t\to W(t,v)$ is linear,
by denoting the corresponding matrix by $W(t)$, 
we may write $W(t,v)=W(t)v$.
Of course, $W(0)=0, W'(0)=I$.
By the assumption that there are no cut-locus on $\{\cyx(t)\}$,
$W(t)$ is invertible linear map for all $0<t\le 1$ and
$J(t,v)=W(t)W(1)^{-1}v$ is the solution to
$$
J''(t,v)+\Rreverse(t)J(t,v)=0,~~J(0,v)=0,~J(1,v)=v
$$
and
$(\nabla^2k(\cyx))_1(v,v)=(J'(1,v),v)=(W'(1)W(1)v,v)$.
Let $0<T\le 1$.
We can obtain explicit form of the Jacobi field along
$\cyx(t)$~$(0\le t\le T)$ with given terminal value at $T$ using
$W$.
Let $J_T(t,v)=W(Tt)W(T)^{-1}v$.
Then $J_T(t,v)$~$0\le t\le 1$ satisfies 
the Jacobi equation
\begin{equation}
J_T''(t)+\Rreverse(tT)T^2J_T(t)=0,~~
J_T(0)=0,~ J_T(1)=v.
\end{equation}
Hence
$\overline{\nabla^2k(\cyx)}_{t}(v,v)=t\left(W'(t)W(t)^{-1}v,v\right)$.
Therefore the Hessian of $k$ at $\cyx(t)$ is given by
$$
\overline{\nabla^2k(\cyx)}_{t}=
tW'(t)W(t)^{-1}.
$$
This can be checked by the following argument.
It suffices to show that
$A(t)=tW'(t)W(t)^{-1}$ is a symmetric operator for $0<t\le 1$.
We have
\begin{eqnarray}
A'(t)&=&W'(t)W(t)^{-1}+tW''(t)W(t)^{-1}
-tW'(t)W(t)^{-1}W'(t)W(t)^{-1}\nonumber\\
&=&-t\Rreverse(t)-\frac{A(t)^2}{t}+\frac{A(t)}{t}.\label{eq for A}
\end{eqnarray}
Note that $\lim_{t\to 0}A(t)=I$.
If we extend $A=A(t)$ by setting $A(0)=I$,
then $A(t)$ is continuously differentiable on
$[0,1]$ and $A'(0)=0$.
One can show this by the equation of $W(t)$.
Let $B(t)=A(t)-A(t)^{\ast}$, where
$A(t)^{\ast}$ denotes the adjoint operator.
Since $\Rreverse(t)$ is a symmetric matrix, $B(t)$ satisfies 
\begin{eqnarray}
B(t)&=&\frac{1}{t}\int_0^t(I-A(s)^{\ast})B(s)ds+
\frac{1}{t}\int_0^tB(s)(I-A(s))ds.
\end{eqnarray}
Noting that
\begin{eqnarray}
\lefteqn{
\frac{1}{t}\int_0^t(I-A(s)^{\ast})B(s)ds}\nonumber\\
& &=
\frac{I-A(t)^{\ast}}{t}\int_0^tB(s)ds
-\frac{1}{t}\int_0^t\left(A(s)^{\ast}\right)'\left(\int_0^sB(r)dr\right)ds
\end{eqnarray}
and using Gronwall's inequality, we obtain 
$B(t)=0$ for all $t$ which we want to show.
Let $f(t)=W(1-t)$.
Then $f$ satisfies
\begin{equation}
f''(t)+R(t)f(t)=0~~~0\le t\le 1,~~ f(1)=0, f'(1)=I
\end{equation}
and we have
$
\overline{\nabla^2k(\cxy)}_t=
-(1-t)f'(t)f(t)^{-1}
$
and $K(t)=-\frac{1}{1-t}\overline{\nabla^2k(\cxy)}_t=
f'(t)f(t)^{-1}$.
Clearly $f(t)$ has the expansion around $t=1$,
$$
f(t)=(1-t)I+\frac{1}{2}(1-t)^2R(1)+(1-t)^2f_2(t),
$$
where 
$f_2(t)$ is a matrix-valued smooth mapping.
Therefore, when $t$ is close to $1$,
\begin{eqnarray}
\overline{\nabla^2k(\cxy)}_t
&=&
\left(I+(1-t)R(1)+2(1-t)f_2(t)-(1-t)^2f_2'(t)\right)\nonumber\\
& &\times \left(I+\frac{1}{2}(1-t)R(1)+
(1-t)f_2(t)\right)^{-1}.\label{expansion of f}
\end{eqnarray}
Let
\begin{equation}
\tilde{K}(t)=K(t)+\frac{1}{1-t}.\label{decomposition of K}
\end{equation}
Using (\ref{expansion of f}), we see that
$\tilde{K}(t)$ 
$(0\le t\le 1)$ is a matrix-valued smooth mapping.
Let $M(t)$ be the solution to
$$
M'(t)=K(t)M(t),~
M(0)=I.
$$
Let $N(t)$ be the solution to
$$
N'(t)=\tilde{K}(t)N(t),~N(0)=I.
$$
Then $M(t)=(1-t)N(t)$.
Note that $\sup_t(\|N(t)\|_{op}+\|N^{-1}(t)\|_{op})<\infty$.
Also we have 
$M(t)=f(t)f(0)^{-1}$.
Let 
$$
L^2_0([0,1]\to T_xM):=L^2\left\{[0,1]\to T_xM~\Bigg |~\int_0^1\varphi(t)dt=0
\right\}.
$$
We may denote this set by
$L^2_0$ for simplicity.
Then
$
\left(U\varphi\right)(t)=\int_0^t\varphi(s)ds
$
is a bijective linear isometry from
$L^2_0([0,1]\to T_xM)$ to $H^1_0([0,1]\to T_xM)$.
Also $U^{-1}h(t)=\dot{h}(t)$.
Let us introduce an operator
\begin{eqnarray}
\left(S\varphi\right)(t)&=&\varphi(t)-f'(t)f(t)^{-1}\int_0^t\varphi(s)ds,\\
{\rm D}(S)&=&L^2_0.
\end{eqnarray}
By Hardy's inequality, $S$ is a bounded linear operator from
$L^2_0$ to $L^2$.
Also we have the following lemma.

\begin{lem}\label{S and T1}
For any $\varphi\in L^2_0$,
\begin{equation}
\|S\varphi\|^2=
((I+T)\varphi,\varphi),\label{S and T}
\end{equation}
where $I$ denotes the identity operator on $L^2_0$ and
\begin{equation}
T\varphi(t)=-\int_t^1R(s)\int_0^s\varphi(u)duds+
\int_0^1\int_t^1R(s)\int_0^s\varphi(u)dudsdt.
\end{equation}
Also we have
\begin{equation}
(D_0^2E)(\cxy)=U\left(I+T\right)U^{-1},
\end{equation}
where $E$ is the energy function of the path~$(\ref{energy function})$.
\end{lem}

\begin{proof}
We calculate $\|S\varphi\|^2$.
Using $\lim_{t\to 1}\frac{1}{1-t}|\int_t^1\varphi(s)ds|^2=0$
for any $\varphi\in L^2$,
$f''(t)=-R(t)f(t)$ and $(f(t)^{-1})'=-f(t)^{-1}f'(t)f(t)^{-1}$, 
we have
\begin{eqnarray}
\|S\varphi\|^2&=&\|\varphi\|^2-2\int_0^1\left(
f'(t)f(t)^{-1}\int_0^t\varphi(s)ds,\varphi(t)\right)dt\nonumber\\
& &+\int_0^1\left|f'(t)f(t)^{-1}\int_0^t\varphi(s)ds\right|^2dt\nonumber\\
&=&\|\varphi\|^2+\int_0^1\left(
\left(f'(t)f(t)^{-1}\right)'\int_0^t\varphi(s)ds,
\int_0^t\varphi(s)ds\right)dt\nonumber\\
& &+\int_0^1\left|f'(t)f(t)^{-1}\int_0^t\varphi(s)ds\right|^2dt\nonumber\\
&=&
\|\varphi\|^2-\int_0^1
\left(R(t)\int_0^t\varphi(s)ds,\int_0^t\varphi(s)ds\right)
dt\nonumber\\
&=&\left((I+T)\varphi,\varphi\right).\label{hessian and R}
\end{eqnarray}
Finally, it is well-known that
$(D_0^2E)(\cxy)(U\varphi,U\varphi)$
is equal to $\left((I+T)\varphi,\varphi\right)$.
\end{proof}

Let
\begin{equation}
(S_2\varphi)(t)=\varphi(t)+f'(t)\int_0^tf(s)^{-1}\varphi(s)ds.
\end{equation}
Then again by Hardy's inequality
$S_2$ is a bounded linear operator on $L^2$.
Moreover it is easy to see that
$\Im(S_2)\subset L^2_0$,
$SS_2=I_{L^2}$ and $S_2S=I_{L^2_0}$.
Therefore, $S_2=S^{-1}$ and
$\Im(S)=L^2$.
Moreover we have 
$S^{\ast}S=I+T$ by (\ref{S and T}).
Note that by identifying the dual space of a Hilbert space
with the Hilbert space itself using Riesz's theorem,
we view $S^{\ast} : (L^2)^{\ast} \to (L^2_0)^{\ast}$ as
the operator from $L^2$ to $L^2_0$.
We have the following explicit expression
of $S^{-1}$, $S^{\ast}$ and $(S^{-1})^{\ast}$.

\begin{lem}\label{S explicit form}
$(1)$~$S^{-1} : L^2\to L^2_0$, 
$S^{\ast} : L^2\to L^2_0$ are bijective linear isometries and
we have for any $\varphi\in L^2$,
\begin{eqnarray}
\left(S^{-1}\varphi\right)(t)&=&
\varphi(t)+f'(t)\int_0^tf(s)^{-1}\varphi(s)ds\\
\left(S^{\ast}\varphi\right)(t)&=&
\varphi(t)-\int_0^1\varphi(t)dt+\int_0^tf'(s)f(s)^{-1}\varphi(s)ds
\nonumber\\
& &-\int_0^1\left(\int_0^tf'(s)f(s)^{-1}\varphi(s)ds\right)dt.
\end{eqnarray}

\noindent
$(2)$ $(S^{-1})^{\ast}$ is a bijective linear isometry 
from $L^2_0$ to $L^2$.
If we define $(S^{-1})^{\ast}$ is equal to $0$ on the subset
of constant functions, then for any
$\varphi\in L^2$,
\begin{eqnarray}
\left((S^{-1})^{\ast}\varphi\right)(t)&=&
\varphi(t)+
\left(f(t)^{\ast}\right)^{-1}\int_t^1
f(s)^{\ast}f'(s)f(s)^{-1}\varphi(s)ds.\label{Sinverseast}
\end{eqnarray}
Also $(S^{-1})^{\ast}\varphi$ can be written using $M(t)$ and $K(t)$
as
\begin{equation}
\left((S^{-1})^{\ast}\varphi\right)(t)=
\varphi(t)+
(M(t)^{\ast})^{-1}\int_t^1M(s)^{\ast}K(s)\varphi(s)ds.
\end{equation}
\end{lem}

\begin{proof}
All the calculation are almost similar and so
we show how to calculate $(S^{-1})^{\ast}$ only.
Using $(f'(t)f(t)^{-1})^{\ast}=f'(t)f(t)^{-1}$, 
we have for $\varphi\in L^2$ with ${\mathrm supp}~\varphi\subset
(0,1)$ and $\psi\in L^2$,
\begin{eqnarray}
\lefteqn{\left(S^{-1}\varphi,\psi\right)_{L^2}}\nonumber\\
& &=(\varphi,\psi)-\int_0^1
\left(\int_0^tf(s)^{-1}\varphi(s)ds, 
\left(
\int_t^1f(s)^{\ast}f'(s)f(s)^{-1}\psi(s)ds\right)'
\right)dt\nonumber\\
&=&(\varphi,\psi)+
\int_0^1\left(
\varphi(t), \left(f(t)^{-1}\right)^{\ast}
\int_t^1f(s)^{\ast}f'(s)f(s)^{-1}\psi(s)ds
\right)dt.
\end{eqnarray}
This shows (\ref{Sinverseast}).
Also $(S^{-1})^{\ast}\mbox{const}=0$ follows from
$f(1)=0$.
\end{proof}

We summarize the relation between $S$ and $T$ 
in the proposition below.

\begin{pro}\label{S and T2}
$(1)$~
$I+T=S^{\ast}S$,
$(S^{-1})^{\ast}(I+T)=S$
and $(I+T)^{-1}=S^{-1}(S^{-1})^{\ast}$.

$(2)$~
\begin{equation}
\inf\sigma(I+T)=\inf
\left\{\|S\varphi\|^2~|~\|\varphi\|_{L^2}=1,
\varphi\in L^2_0\right\}=
\frac{1}{\|(S^{-1})^{\ast}\|_{op}^2}.
\end{equation}
\end{pro}

\begin{proof}
$I+T=S^{\ast}S$ follows from Lemma~\ref{S and T1}.
$(I+T)^{-1}=S^{-1}(S^{-1})^{\ast}$ follows from
$(S^{-1})^{\ast}=(S^{\ast})^{-1}$.
(2) follows from (1).
\end{proof}

The relation $\inf\sigma((I+T))\le \frac{1}{\|(S^{-1})^{\ast}\|_{op}^2}$
plays important role to prove the lower bound estimate
in Theorem~\ref{main theorem}, while
the relation $\inf\sigma(I+T)\ge\inf
\left\{\|S\varphi\|^2~|~\|\varphi\|_{L^2}=1,
\varphi\in L^2_0\right\}$ is used to prove the upper bound estimate.

\section{Proof of Main Theorem}

\begin{lem}\label{ground state}
Let $\Omega$ be the normalized ground state function
of $-L_{\la}$.
Then $\|\Omega-1\|_{L^2(P_{x,y}(M),\nu^{\la}_{x,y})}\le C e^{-C'\la}$,
where $C,C'$ are positive constants.
\end{lem}

\begin{proof}
Let $\varphi_{\delta}(\gamma)=\chi(\max_{0\le t\le 1}d(\gamma(t),\cxy(t)))$,
where $\chi$ is a non-negative smooth function such that 
$\chi(u)=1$ for $|u|\le \delta$ and $\chi(u)=0$
for $|u|\ge 2\delta$.
Here $\delta$ is a sufficiently small positive number.
Then $\|\varphi_{\delta}\|_{L^2(\nu^{\la}_{x,y})}\ge 1-Ce^{-C'\la}$
and
$
\|D_0\varphi_{\delta}\|_{L^2(\nu^{\la}_{x,y})}
\le Ce^{-C'\la}.
$
Hence 
\begin{equation}
e^{\la}_{Dir,1,P_{x,y}({\mathcal D})}\le 
Ce^{-\la C'}.
\end{equation}
By the COH formula,
$$
\|\Omega-\left(\Omega,1\right)
_{L^2({\nu}^{\la}_{x,y})}\|_{L^2(\nu^{\la}_{x,y})}
\le Ce^{-C'\la}.
$$
This implies 
$$
1-\left(\Omega,1\right)_{L^2({\nu}^{\la}_{x,y})}^2
=\left(\Omega,\Omega-\left(\Omega,1\right)_{L^2({\nu}^{\la}_{x,y})}\right)
_{L^2(\nu^{\la}_{x,y})}
\le Ce^{-C'\la}
$$
which shows
$\|\Omega-1\|_{L^2(P_{x,y}(M),\nu^{\la}_{x,y})}^2\le
2Ce^{-C'\la}$.
\end{proof}

\begin{lem}\label{perturbation of M}
As already defined, let
$K(t)=-\frac{\overline{\nabla^2k(c_{x,y})}_t}{1-t}$.
Also we consider a perturbation
of $K(t)$ such that
$$
K_{\ep}(t)=K(t)+\frac{C_{\ep}(t)}{(1-t)^{\delta}},
$$
where $0<\delta<1$ is a constant and
$C_{\ep}(t)$~$(0\le \ep\le 1)$ be a symmetric matrices valued
function satisfying $\sup_{t}\|C_{\ep}(t)\|\le \ep$.
Let
$M_{\ep}(t)$ be the solution to
\begin{eqnarray}
M_{\ep}'(t)&=&K_{\ep}(t)M_{\ep}(t)\\
M_{\ep}(0)&=&I
\end{eqnarray}
Define
\begin{eqnarray}
J_{\ep}\varphi(t)&=&
(M_{\ep}(t)^{\ast})^{-1}\int_t^1
M_{\ep}(s)^{\ast}K_{\ep}(s)\varphi(s)ds.
\end{eqnarray}
Then for sufficiently small $\ep$,
there exists a positive constant $C$ which is independent of
$\ep$ such that
\begin{equation}
\|J_{\ep}-J_0\|_{op}\le C\ep.\label{difference of J}
\end{equation}
Also $J_0=(S^{-1})^{\ast}$ holds.
\end{lem}

\begin{proof}
First, we recall that
$K(t)$ is the sum of $-\frac{1}{1-t}I$ and $\tilde{K}(t)$
as in (\ref{decomposition of K}).
Taking this into account, we rewrite
$$
K_{\ep}(t)=-\frac{1}{1-t}+\tilde{K}_{\ep}(t),
$$
where $\tilde{K}_{\ep}(t)=\tilde{K}(t)+\frac{C_{\ep}(t)}{(1-t)^{\delta}}$.
Let $N_{\ep}(t)$ be the solution to
\begin{equation}
N_{\ep}'(t)=\tilde{K}_{\ep}(t)N_{\ep}(t),~ N_{\ep}(0)=I.
\end{equation}
Note that $\tilde{K}_0(t)=\tilde{K}(t)$ and
$N_{0}(t)=N(t)$.
Then $M_{\ep}(t)=(1-t)N_{\ep}(t)$.
To estimate $J_{\ep}-J_0$, we need to estimate
$N_{\ep}-N_0$.
Note that
$$
N_{\ep}(t)=N_0(t)\left(I+\int_0^tN_0(s)^{-1}
\frac{C_{\ep}(s)}{(1-s)^{\delta}}ds\right).
$$
This implies
$$
\sup_t|N_{\ep}(t)-N_0(t)|\le C\ep.
$$
The constant $C$ depends $\delta$.
$N_{\ep}(t)^{-1}-N_{0}(t)^{-1}$
has also similar estimates.
By this estimate and Hardy's inequality, we complete the
proof of (\ref{difference of J}).
\end{proof}

Let us apply the lemma above in the case where
$K_{\ep}(t)=K(\gamma)_{\la,t}$.
In this case, we have
\begin{eqnarray}
\lefteqn{K(\gamma)_{\la,t}}\nonumber\\
& &=K(t)+
\frac{1}{1-t}
\left(\frac{1-t}{\la}\overline{\nabla^2\log
p\left(\frac{1-t}{\la},y,\gamma\right)}_t+
\overline{\nabla^2k(\cxy)}_t\right)-
\frac{1}{2\la}\Ricbar\nonumber\\
& &=K(t)+
\frac{1}{1-t}\left(\frac{1-t}{\la}\overline{\nabla^2\log
p\left(\frac{1-t}{\la},y,\gamma\right)}_t+
\overline{\nabla^2k(\gamma)}_t\right)\nonumber\\
& &\quad +\frac{1}{1-t}\left(\overline{\nabla^2k(\cxy)}_t-
\overline{\nabla^2k(\gamma)}_t\right)
-\frac{1}{2\la}\Ricbar.
\end{eqnarray}
Hence
\begin{eqnarray}
C_{\ep}(t)&=&
\frac{1}{(1-t)^{1-\delta}}
\left(
\frac{1-t}{\la}\overline{\nabla^2\log
p\left(\frac{1-t}{\la},y,\gamma\right)}_t
+\overline{\nabla^2k(\gamma)}_t\right)\nonumber\\
& &+\frac{1}{(1-t)^{1-\delta}}
\left(\overline{\nabla^2k(\cxy)}_t-\overline{\nabla^2k(\gamma)}_t\right)
\nonumber\\
& &-\frac{(1-t)^{\delta}}{2\la}\Ricbar.\label{Cept}
\end{eqnarray}

Now, we are in a position to prove our main theorem.

\begin{proof}[Proof of Theorem~$\ref{main theorem}$]
Let $l_{\xi}(t)=t\xi$.
Let $\chi$ be a non-negative smooth function such that
$\chi(t)=1$ for $t\le 1$ and
$\chi(t)=0$ for $t\ge 2$.
Let $\kappa>0$ and
$
\chi_{1,\kappa}(\gamma)=
\chi(\kappa^{-1}\|\overline{b-l_{\xi}}\|_{T^2,B,2m,\theta}^{2m})
$
and $\chi_{2,\kappa}(\gamma)=\left(1-\chi_{1,\kappa}(\gamma)^2\right)^{1/2}$.
Here $\|~\|_{T^2,B,2m,\theta}$~$(0<\theta<1, 
\mbox{$m$ is a large positive integer})$ denotes the norm for
the Brownian rough path $\overline{b-l_{\xi}}$ over $b-l_{\xi}$.
See Definition~7.2 in \cite{aida-semiclassical}.
Actually, we need quasi-sure version of Brownian rough path as in
Theorem~3.1 in \cite{aida-loop group} because
we are considering the pinned Brownian motion measure.
By Lemma~7.11 in \cite{aida-semiclassical}, 
there exists a positive constant $C_{\kappa}$ such that
\begin{equation}
|D_0\chi_{1,\kappa}(\gamma)|_{H^1_0}+
|D_0\chi_{2,\kappa}(\gamma)|_{H^1_0}\le C_{\kappa}
\qquad \nu_{x,y}^{\la}-a.s. \gamma.
\label{estimate on cut-off}
\end{equation}
First we prove the upper bound estimate.
Let us fix a positive number $\ep>0$.
Let us choose $\varphi_{\ep}\in L^2_0$ with $\|\varphi_{\ep}\|=1$ 
such that
\begin{equation}
\max\left(
\left|\|(I+T)\varphi_{\ep}\|-e_0\right|,
\left|\|S\varphi_{\ep}\|^2-e_0\right|
\right)\le \ep.\label{varphiep}
\end{equation}
This is possible because of 
Lemma~\ref{S and T1} and Proposition~\ref{S and T2}.
Let $\psi_{\ep}=U\varphi_{\ep}$.
Let
$F_{\ep}(\gamma)=\sqrt{\la}\left(b-l_{\xi},\psi_{\ep}\right)_{H^1_0}$
and $\tilde{F}_{\ep}=F_{\ep}\chi_{1,\kappa}$.
Note that
$$
(D_0)_{h}b(t)=
h(t)+\int_0^t\int_0^s\overline{R(\gamma)}_u(h(s),\circ db(u))(\circ db(s)).
$$
Hence
\begin{eqnarray}
D_0F_{\ep}(\gamma)_t'&=&
\sqrt{\la}\psi_{\ep}'(t)+
\sqrt{\la}\int_t^1\overline{R(\gamma)}_s\left(
\int_s^1\varphi_{\ep}(r)db^i(r),\circ db(s)\right)\ep_i\nonumber\\
& &-\sqrt{\la}t\int_0^1\int_t^1\overline{R(\gamma)}_s\left(
\int_s^1\varphi_{\ep}(r)db^i(r),\circ db(s)\right)\ep_idt\nonumber\\
&=&\sqrt{\la}\psi_{\ep}'(t)-
\sqrt{\la}\int_t^1R(s)\left(
\psi_{\ep}(s),\xi\right)\xi ds\nonumber\\
& &+\sqrt{\la}\int_0^1\int_t^1R(s)\left(
\psi_{\ep}(s),\xi\right)\xi dsdt
+I(\la)\nonumber\\
&=&\sqrt{\la}(I+T)(\varphi_{\ep})(t)+I(\la).
\end{eqnarray}
By applying the Cameron-Martin formula to the translation
$b\to b+l_{\xi}$ and using the similar argument to
page 33 in \cite{watanabe},
we obtain 
$E^{\nu^{\la}_{x,y}}[|I(\la)|^p]\le C_{p,\kappa}$.
Here the constant $C_{p,\kappa}$ is independent of
$\la$.
Also we have
$E[|D_0\chi_{1,\kappa}|_{H^1_0}^2]\le Ce^{-C'\la}$
for some $C,C'>0$ which follows from the Schilder type
large deviation principle.
This implies
\begin{equation}
\|D_0\tilde{F}_{\ep}\|_{L^2(\nu^{\la}_{x,y})}^2
=\la\|(I+T)\varphi_{\ep}\|_{L^2}^2+
o(\la).
\end{equation}
Combining $\|D_0\Omega\|\le Ce^{-C'\la}$,
we obtain
\begin{equation}
\|D_0\tilde{F}_{\ep}-
\left(\tilde{F}_{\ep},\Omega\right)D_0\Omega\|_{L^2(\nu^{\la}_{x,y})}^2
=\la\|(I+T)\varphi_{\ep}\|_{L^2}^2+
o(\la).\label{Dirichlet norm}
\end{equation}
On the other hand, by the continuity theorem of rough path,
there exists $\ep'>0$ such that
for 
$\|\overline{b-l_{\xi}}\|_{T^2,B,2m,\theta}\le \ep'$, 
\begin{equation}
\|\overline{\nabla^2k(\cxy)}_t-\nabla^2k(\gamma)_t\|
\le \ep |1-t|^{\theta/2}.
\end{equation}
By Lemma~3.3 in \cite{gong-ma}, 
for any compact subset $F\subset {\rm Cut}(y)^c$,
\begin{eqnarray}
\sup_{z\in F}\left|
t\nabla^2\log p(t,z,y)-\nabla^2k(z)\right|
\le C_F t^{1/2}.
\end{eqnarray}
Therefore, by setting $\kappa$ to be sufficiently small,
$C_{\ep}(t)$ which is defined in (\ref{Cept})
satisfies the assumption in Lemma~\ref{perturbation of M}
for certain $\delta>1/2$.
Hence, by Lemma~\ref{perturbation of M},
by taking $\kappa$ to be sufficiently small,
\begin{eqnarray}
A(\gamma)_{\la}(D_0\tilde{F}_{\ep}(\gamma)')_t&=&
\left(S^{-1}\right)^{\ast}\left(D_0\tilde{F}_{\ep}(\gamma)'\right)_t
+(J(\gamma)_{\la}-J_0)(D_0\tilde{F}_{\ep}(\gamma)')_t
\nonumber\\
&=&
\sqrt{\la}S\varphi_{\ep}(t)+I_2(\la),
\end{eqnarray}
where $|I_2(\la)|_{L^2([0,1])}\le C \ep \sqrt{\la}$.
Also we have used that $(S^{-1})^{\ast}(I+T)=S$.
This estimate, (\ref{estimate on cut-off})
and the COH formula (\ref{coh1}) implies that
\begin{equation}
\|\tilde{F}_{\ep}-E^{\nu^{\la}_{x,y}}
[\tilde{F}_{\ep}]\|_{L^2(\nu^{\la}_{x,y})}^2
=\|S\varphi_{\ep}\|^2+C \ep.
\end{equation}
Using Lemma~\ref{ground state},
\begin{equation}
\|\tilde{F}_{\ep}-
\left(\tilde{F}_{\ep},\Omega\right)\Omega\|_{L^2(\nu^{\la}_{x,y})}^2
=\|S\varphi_{\ep}\|^2+C \ep+Ce^{-C'\la}.\label{L2 norm}
\end{equation}
By using the estimates (\ref{Dirichlet norm}), (\ref{L2 norm})
and (\ref{varphiep}),
we complete the proof of the upper bound.
We prove lower bound.
Take $F\in H^{1,2}_0(P_{x,y}({\mathcal D}))$
such that $\|F\|_{L^2(\bar{\nu}^{\la}_{x,y})}=1$ and
$(F,\Omega)=0$.
By the IMS localization formula,
\begin{equation}
{\mathcal E}(F,F)=
\sum_{i=1,2}{\mathcal E}(F\chi_{i,\kappa},F\chi_{i,\kappa})
-\sum_{i=1,2}
E^{\bar{\nu}^{\la}_{x,y}}[|D_0\chi_{i,\kappa}|^2F^2].
\end{equation}
By Lemma~\ref{ground state},
\begin{equation}
|E^{\nu^{\la}_{x,y}}[F]|
=|E^{\nu^{\la}_{x,y}}[F(1-\Omega)]|
\le \|1-\Omega\|\le Ce^{-C'\la}.
\end{equation}
Again by taking $\kappa$ to be sufficiently small,
by Lemma~\ref{perturbation of M} and the COH formula
\begin{eqnarray}
\lefteqn{
\|F\chi_{1,\kappa}-E^{\nu^{\la}_{x,y}}[F\chi_{1,\kappa}]\|_{L^2(P_{x,y}(M))}^2}
\nonumber\\
& &\le
\frac{\left(\|(S^{-1})^{\ast}\|_{op}+C\ep\right)^2}{\la}
E^{\nu^{\la}_{x,y}}\left[
|D_0(F\chi_{1,\kappa})|^2
\right].
\end{eqnarray}
Consequently,
we have
\begin{equation}
\|F\chi_{1,\kappa}\|_{L^2(P_{x,y}({\mathcal D}))}^2
\le
\frac{\left(\|(S^{-1})^{\ast}\|_{op}+C\ep\right)^2}{\la}
E^{\nu^{\la}_{x,y}}\left[
|D_0(F\chi_{1,\kappa})|^2
\right]+Ce^{-C'\la}.\label{Fchi1}
\end{equation}
Now we estimate the Dirichlet norm of $F\chi_{2,\kappa}$.
The log-Sobolev inequality (\ref{lsi1}) implies that
there exists a positive constant $C$ such that
for any $F\in H^{1,2}_0(P_{x,y}({\mathcal D}))$ and
bounded measurable function $V$,
\begin{eqnarray}
{\mathcal E}(F,F)+E^{\bar{\nu}^{\la}_{x,y}}
\left[\la^2 VF^2\right]
\ge
-\frac{\la}{C}\log E^{\bar{\nu}^{\la}_{x,y}}
\left[e^{-C\la V}\right]\|F\|_{L^2(\bar{\nu}^{\la}_{x,y})}^2.
\label{GNS}
\end{eqnarray}
We apply this inequality.
Take a non-negative smooth function $\tilde{\chi}$ on $\RR$ such that
$\tilde{\chi}=1$ on a neighborhood of 
the support of $1-\chi^2$ and $\tilde{\chi}=0$ on
a neighborhood of $0$.
Let $\delta$ be a sufficiently small positive number $\delta$
and define
$$
V(\gamma)=\delta \tilde{\chi}
(\kappa^{-1}\|\overline{b-l_{\xi}}\|_{T^2,B,2m,\theta}^{2m}).
$$
By (\ref{GNS}), there exists $\delta'>0$ such that
\begin{eqnarray}
\lefteqn{{\mathcal E}(F\chi_{2,\kappa},F\chi_{2,\kappa})}\nonumber\\
& &=
{\mathcal E}(F\chi_{2,\kappa},F\chi_{2,\kappa})-\la^2
E^{\bar{\nu}^{\la}_{x,y}}\left[V(F\chi_{2,\kappa})^2\right]
+\la^2
E^{\bar{\nu}^{\la}_{x,y}}\left[V(F\chi_{2,\kappa})^2\right]\nonumber\\
&\ge&
-\frac{\la}{C} \log E^{\bar{\nu}^{\la}_{x,y}}
\left[e^{C\la V}\right]\|F\chi_{2,\kappa}\|^2
+\la^2\delta\|F\chi_{2,\kappa}\|^2\nonumber\\
&\ge&
-\frac{\la}{C}\log\left(1+e^{-\la\delta'}\right)\|F\chi_{2,\kappa}\|^2
+\la^2\delta\|F\chi_{2,\kappa}\|^2\nonumber\\
&\ge&(\la^2\delta-C\la e^{-\la\delta'})\|F\chi_{2,\kappa}\|^2.\label{Fchi2}
\end{eqnarray}
By the estimates (\ref{estimate on cut-off}), (\ref{Fchi1}),
(\ref{Fchi2}) and by Proposition~\ref{S and T2},
we complete the proof.
\end{proof}


\begin{thebibliography}{99}
%
\bibitem{aida-coh}
S.~Aida,
Logarithmic derivatives of heat kernels and logarithmic Sobolev inequalities 
with unbounded diffusion coefficients on loop spaces.  
J. Funct. Anal.  174  (2000),  no. 2, 430--477. 
%
\bibitem{aida-semiclassical}
S.~Aida,
Semi-classical limit of the bottom of spectrum of a Schr\"odinger operator on 
a path space over a compact Riemannian manifold.  
J. Funct. Anal.  251  (2007),  no. 1, 59--121. 
%
\bibitem{aida-coh2}
S.~Aida,
COH formula and Dirichlet Laplacians
on small domains of pinned path spaces,
to appear in Contemporary Mathematics,
as a proceedings of the workshop 
"Concentration, Functional Inequalities, and Isoperimetry". 
%
\bibitem{aida-loop group}
S.~Aida,
Vanishing of one dimensional $L^2$-cohomologies of
loop groups, submitted, 2009.
%
\bibitem{chl}
M.~Captaine, E.~Hsu and M.~Ledoux,
Martingale representation and a simple proof of 
logarithmic Sobolev inequalities on path spaces.  
Electron. Comm. Probab.  2  (1997), 71--81 
%
\bibitem{clw1}
X.~Chen, X.-M.~Li and B.~Wu,
A Poincaré inequality on loop spaces.  J. Funct. Anal.  259  (2010),  
no. 6, 1421–1442.
%
\bibitem{clw2}
X.~Chen, X.-M.~Li and B.~Wu,
A concrete estimate for the weak Poincar\'e
inequality on loop space,
Probab.Theory Relat. Fields,
DOI: 10.1007/s00440-010-0308-5
%
\bibitem{eberle1}
A.~Eberle, 
Absence of spectral gaps on a class of loop spaces.  
J. Math. Pures Appl. (9)  81  (2002),  no. 10, 915--955. 
%
\bibitem{eberle2}
A.~Eberle,
Spectral gaps on discretized loop spaces.  
Infin. Dimens. Anal. Quantum Probab. Relat. Top.  6  (2003),  no. 2, 265--300.
%
\bibitem{eberle3}
A.~Eberle,
Local spectral gaps on loop spaces.  
J. Math. Pures Appl. (9)  82  (2003),  no. 3, 313--365.
%
\bibitem{gong-ma}
F.~Gong and Z.~Ma,
The log-Sobolev inequality on loop space over a compact Riemannian manifold.  
J. Funct. Anal.  157  (1998),  no. 2, 599--623. 
%
\bibitem{ms}
P.~Malliavin and D.W.~Stroock,
Short time behavior of the heat kernel and its logarithmic derivatives.  
J. Differential Geom.  44  (1996),  no. 3, 550--570.
%
\bibitem{watanabe}
S.~Watanabe,
Analysis of Wiener functionals (Malliavin calculus) and its applications
to heat kernels,
Ann. of Probab. Vol. 15, No.1, (1987), 1--39.
%
\end{thebibliography}
\end{document}